\documentclass[11pt,leqno]{amsart}
 \usepackage{graphicx}    
 \usepackage{amsmath}
\usepackage{amssymb}
\usepackage[T1]{fontenc}
\usepackage[latin1]{inputenc}
\usepackage[english]{babel}
\usepackage[arrow, matrix, curve]{xy}
\usepackage{amsthm}
\usepackage{amsmath,amscd}
\usepackage{mathrsfs}
\usepackage{epic,eepic}
\numberwithin{equation}{section}

\DeclareMathOperator\GL{GL}

\DeclareMathOperator\Spec{Spec}
\DeclareMathOperator\Spf{Spf}

\DeclareMathOperator\Spa{Spa}

\DeclareMathOperator\Fil{Fil}

\DeclareMathOperator\rig{rig}

\renewcommand{\phi}{\varphi}

\newcommand{\Acal}{\mathscr{A}}
\newcommand{\Bcal}{\mathscr{B}}
\newcommand{\Ccal}{\mathcal{C}}
\newcommand{\Dcal}{\mathcal{D}}

\newcommand{\Hcal}{\mathcal{H}}
\newcommand{\Mcal}{\mathcal{M}}
\newcommand{\Ncal}{\mathcal{N}}
\newcommand{\Ocal}{\mathcal{O}}
\newcommand{\Rcal}{\mathcal{R}}
\newcommand{\Scal}{\mathcal{S}}
\newcommand{\Tcal}{\mathcal{T}}

\newcommand{\Xcal}{\mathcal{X}}
\newcommand{\Vcal}{\mathcal{V}}
\newcommand{\Wcal}{\mathcal{W}}

\newcommand{\Zcal}{\mathcal{Z}}

\newcommand{\Q}{\mathbb{Q}}
\newcommand{\R}{\mathbb{R}}
\newcommand{\Z}{\mathbb{Z}}
\newcommand{\C}{\mathbb{C}}

\newcommand{\Abb}{\mathbb{A}}
\newcommand{\boldB}{\mathbb{B}}
\newcommand{\Fbb}{\mathbb{F}}
\newcommand{\Gbb}{\mathbb{G}}

\newcommand{\Ubb}{\mathbb{U}}

\newcommand{\Nfrak}{\mathfrak{N}}

\newcommand{\Xfrak}{\mathfrak{X}}

\newtheorem{theo}{Theorem}[section]
\newtheorem{lem}[theo]{Lemma}
\newtheorem{prop}[theo]{Proposition}
\newtheorem{cor}[theo]{Corollary}
\newtheorem{conj}[theo]{Conjecture}
\theoremstyle{remark}
\newtheorem{rem}[theo]{Remark}
\theoremstyle{remark}

\theoremstyle{definition}
\newtheorem{defn}[theo]{Definition}

\begin{document}

\title[Trianguline representations and finite slope spaces  ]{Families of trianguline representations and finite slope spaces}
\author[E. Hellmann]{Eugen Hellmann}\thanks{Mathematisches Institut der Universit\"at Bonn, Endenicher Allee 60, 53115 Bonn \\ E-mail: hellmann@math.uni-bonn.de}
\begin{abstract}
We apply the theory of families of $(\phi,\Gamma)$-modules to trianguline families as defined by Chenevier. This yields a new definition of Kisin's \emph{finite slope} subspace as well as higher dimensional analogues. 
Especially we show that these finite slope spaces contain eigenvarieties for unitary groups as closed subspaces. 
This implies that the representations arising from overconvergent $p$-adic automorphic forms on certain unitary groups are trianguline when restricted to the local Galois group.
\end{abstract}
\maketitle

\section{Introduction}

The idea of a \emph{finite slope space} follows an construction of Kisin \cite{KisinsXfs} who aims at giving a new description of the eigencurve using families of $p$-adic Galois representation instead of overconvergent $p$-adic modular forms. The starting point for this construction is the observation that the Galois representations arising from overconvergent $p$-adic eigenforms $f$ admit one crystalline period after restriction to a decomposition group at $p$. Even better, the crystalline Frobenius acts on this period via multiplication with the $U_p$-eigenvalue of $f$. The finite slope space of Kisin \cite{KisinsXfs} is a space parametrizing deformations $\rho$ of a fixed residual representation $\bar\rho$ of a global Galois group, together with one crystalline period of the restriction $\rho|_{\mathcal{G}_p}$ to the local Galois group at $p$. 

The existence of a crystalline period can be  generalized to the notion of a trianguline $(\phi,\Gamma)$-module, as defined by Colmez \cite{Colmez}, see also Berger's survey article \cite{Berger}. These objects form a special class of $(\phi,\Gamma)$-modules over the Robba ring, namely those $(\phi,\Gamma)$-modules that are a successive extensions of one-dimensional ones. In this article we give a new definition of a finite slope space using families of trianguline $(\phi,\Gamma)$-modules as defined by Chenevier \cite{Chenevier}. One of the main steps is to use the construction of a family of Galois representations inside a family of $(\phi,\Gamma)$-modules. This construction is carried out in our paper \cite{families} building on work of Kedlaya and Liu \cite{KedlayaLiu}. 

Finally, we will construct additional structures on the finite slope space that look similar to the structures on eigenvarieties for unitary groups, as studied by Bellaiche and Chenevier in \cite{BellaicheChenevier}.  These extra structures will yield a map from the eigenvariety to our finite slope space, which turns out to be a closed embedding. 
We conjecture that this map induces an isomorphism of an eigenvariety for a unitary group onto a union of irreducible components of (some modification of) the finite slope space. 


We show that this conjecture is equivalent to a conjecture of Bellaiche and Chenevier concerning the local geometry of an eigenvariety at non-critically refined points. And, even better, it is enough to prove the conjecture of \cite{BellaicheChenevier} for just one non-critical crystalline point in each irreducible component. 

More precisely, our results are as follows. Let $\mathcal{G}_p={\rm Gal}(\bar\Q_p/\Q_p)$. We write $\Tcal$ for the space of continuous characters of $\Q_p^\times$ and let $\Tcal_d^{\rm reg}\subset \Tcal^d$ be the set of characters satisfying a certain regularity condition (compare $(\ref{Qpcharacters})$  for the precise definition). Let $\bar\tau:\mathcal{G}_p\rightarrow\Fbb$ be a continuous pseudo-character with values in a finite field $\Fbb$ and write $\Xfrak_{\bar\tau}$ for the generic fiber of the universal deformation ring of $\bar\tau$.  
\begin{theo} Let $p>d$ and let $\bar\tau:\mathcal{G}_p\rightarrow \Fbb$ be a continuous pseudo-character of dimension $d$.\\ 
\noindent {\rm (i)} There is a space $\Scal^{\rm ns}(\bar\tau)$ of trianguline $\mathcal{G}_p$-representations with regular parameters {\rm (}i.e. with a morphism $\Scal^{\rm ns}(\bar\tau)\rightarrow \Tcal_d^{\rm reg}${\rm )} such that the trace of the trianguline family reduces to $\bar\tau$ modulo the ideal of topologically nilpotent elements {\rm (}see $(\ref{reductionmodp})${\rm )}. \\
\noindent {\rm (ii)} The induced morphism $\pi_{\bar\tau}:\Scal^{\rm ns}(\bar\tau)\rightarrow \Xfrak_{\bar\tau}\times\Tcal_d^{\rm reg}$ is finite and injective. 
\end{theo}
\begin{rem}
(i) The morphism is in general not a closed embedding but indeed only finite and injective: The space $\Scal^{\rm ns}(\bar\tau)$ carries a family of Galois representations, while $\Xfrak_{\bar\tau}\times\Tcal_d^{\rm reg}$ carries a family of pseudo-characters. Hence the field extension $k(x)$ over $k(\pi_{\bar\tau}(x))$ is in general not the identity. \\
(ii) It seems that the assumption $p>d$ is not necessary. We make this assumption, as we have to use pseudo-characters. However, using the theory of \emph{determinants} developed by Chenevier in \cite{determinants} it should be possible to remove this assumption. 
\end{rem}

We write $X(\bar\tau)^{\rm reg}$ for the image of $\Scal^{\rm ns}(\bar\tau)$ and $X(\bar\tau)\subset \Xfrak_{\bar\tau}\times\Tcal^d$ for its Zariski-closure. This space is the \emph{local finite slope space}. We show that $X(\bar\tau)$ coincides with Kisin's definition of a finite slope space in the $2$-dimensional case. As an application of the construction, we find that the finite slope space has the following properties.

\begin{cor}
\noindent {\rm (i)} The connected components of $X(\bar\tau)$ are irreducible of dimension $1+\tfrac{d(d+1)}{2}=1+\dim B$, where $B\subset \GL_d$ is a Borel subgroup.\\
\noindent {\rm (ii)} Let $\tilde X(\bar\tau)$ denote the union of irreducible components of $X(\bar\tau)$ that contain at least one crystalline point. Then $\tilde X(\bar\tau)$ is the universal refined family\footnote{in the sense of Bellaiche and Chenevier. See also Definition $\ref{defnrefinedfamily}$.} of $\mathcal{G}_p$-representations. Especially, the crystalline representations are Zariski-dense and accumulation. 
\end{cor}

Let $E$ be a finite extension of $\Q$ and $S$ a finite set of primes of $E$. Let $\mathcal{G}_{E,S}$ denote the Galois group of the maximal extension of $E$ that is unramified outside $S$ and let $\bar\tau:\mathcal{G}_{E,S}\rightarrow \Fbb$ be a pseudo-character. We assume that $p$ completely splits in $E$ and write $\bar\tau_p$ for the restriction of $\bar\tau$ to $\mathcal{G}_p={\rm Gal}(\bar\Q_p/\Q_p)$. Let $X(\bar\tau)$ be the base change of the local finite slope space $X(\bar\tau_p)$ along the canonical map
\[\Xfrak_{\bar\tau}\rightarrow \Xfrak_{\bar\tau_p}\]
between the generic fibers of the universal deformation rings of $\bar\tau$ resp. $\bar\tau_p$. We show that this space admits additional structures that are similar to the structures that one has for eigenvarieties for forms of $\GL_d$. Especially there is an action of an unramified Hecke-algebra $\Hcal_{\rm ur}$ and of the Atkin-Lehner ring $\mathcal{A}_p$. 

To compare our finite slope spaces with eigenvarieties, we follow closely the exposition in the book \cite{BellaicheChenevier} of Bellaiche and Chenevier.
Assume that $E$ is an imaginary quadratic extension of $\Q$ with norm $N:E\rightarrow \Q$. Let $G$ be the definite unitary group attached to the hermitian form $(x_1,\dots, x_d)\mapsto \sum_i N(x_i)$.
We assume that $p$ splits in $E$ and hence $G$ is split at $p$. 
Let $Y$ denote the \emph{eigenvariety} (compare Definition $\ref{defneigenvar}$) associated to a certain set $\Zcal$ (see $(\ref{refinedautomorphic})$) of automorphic representations of $G(\Abb)$. 
\begin{theo}
There is a closed embedding $Y(\bar\tau)\rightarrow X(\bar\tau)$ of the $\bar\tau$-part of the eigenvariety into the finite slope space $X(\bar\tau)$. 
\end{theo}
This result has the following consequences:
\begin{cor} Let $\bar\rho:\mathcal{G}_{E,S}\rightarrow \GL_d(\Fbb)$ be an absolutely irreducible representation with coefficients in a finite extension of $\Fbb_p$ and $p>d$. \\
\noindent {\rm (i)} The Galois-representations associated to the $p$-adic automorphic forms on the $\bar\rho$-part $Y(\bar\rho)$ of the eigenvariety $Y$ are trianguline when restricted to $\mathcal{G}_p$. \\
\noindent {\rm (ii)} Let $Z\subset Y(\bar\rho)$ be an irreducible component such that there exists a point $z\in Z$ such that the $\mathcal{G}_{E,S}$-representation $\rho_z$ at $z$ is absolutely irreducible when restricted to $\mathcal{G}_p$ .Then the family of $\mathcal{G}_{E,S}$-representations on $Y(\bar\rho)$ is a trianguline family over an Zariski-open dense subset of $Z$.
\end{cor}
After the writing of this paper was finished the author was informed that R. Liu \cite{Liu} has obtained similar results as in the first part of this paper, using different methods.

{\bf Acknowledgements:}
I thank G. Chenevier for a conversation which was the starting point of this paper and for his remarks on the first version of this paper. Further I thank M. Rapoport, T. Richarz and P. Scholze for helpful conversations.  
The author was partially supported by the SFB TR 45 of the DFG (German Research Foundation).

\section{Preliminaries}
We recall some definitions and results of our paper \cite{families} as well as some results from Chenevier's paper \cite{Chenevier}.

\subsection{Families of $p$-adic Galois representations and pseudo-characters}

Throughout the first part of the paper we work over the base filed $K=\Q_p$ and write $\mathcal{G}_p={\rm Gal}(\bar\Q_p/\Q_p)$ for its absolute Galois group. Let $X$ be a reduced adic space locally of finite type over $\Q_p$. A family of $\mathcal{G}_p$-representations on $X$ is  vector bundle $\Vcal$ on $X$ which is endowed with a continuous $\Ocal_X$-linear $\mathcal{G}_p$-action. 

For $0\leq r\leq s<1$ let \[\boldB_{[r,s]}=\Spa(\Q_p\langle p^{-s}T, p^r T^{-1}\rangle,\Z_p\langle p^{-s}T, p^r T^{-1}\rangle)\]
denote the closed annulus of inner radius $r$ and outer radius $s$ over $\Q_p$. 
We write \[\Rcal_X=\Bcal_{X,\rig}^\dagger=\lim \limits_{\longrightarrow r} \lim\limits_{\longleftarrow s} \Gamma(X\times \boldB_{[r,s]},\Ocal_{X\times \boldB_{[r,s]}})\] for the relative Robba ring on $X$ which is endowed with an $\Ocal_X$-linear endomorphism $\phi$ and a continuous $\Ocal_X$-linear action of the pro-cyclic group $\Gamma\cong\Z_p^\times$, see \cite[2]{families}. A family of $(\phi,\Gamma)$-modules on $X$ is an $\Rcal_X$-module $D$ which is locally on $X$ free of finite rank over $\Rcal_X$  together with a $\phi$-linear automorphism $\Phi:\phi^\ast D\rightarrow D$ and a semi-linear action of $\Gamma$  commuting with $\Phi$. Recall further that a $(\phi,\Gamma)$-module $D$ over $\Rcal_X$ is called \'etale, if it admits a lattice over the integral subring 
\[\Rcal_X^{\rm int}=\Acal_{X}^\dagger=\lim \limits_{\longrightarrow r} \lim\limits_{\longleftarrow s} \Gamma(X\times \boldB_{[r,s]},\Ocal^+_{X\times \boldB_{[r,s]}}),\]
see also \cite[4]{families} for the definitions. 
By the work of Berger and Colmez \cite{BergerColmez} there is a fully faithful functor
\[\Vcal\longmapsto {\bf D}_{\rm rig}^\dagger(\Vcal)\] from the category of families of $\mathcal{G}_p$-representations on $X$ to the category of families of $(\phi,\Gamma)$-Modules on $X$.  Given a family $D$ of $(\phi,\Gamma)$-modules on $X$, the main result of \cite{families} shows that one can construct a natural open subspace of $X$ over which $D$ is induced by a family of $\mathcal{G}_p$-representations:

\begin{theo}\label{maintheofamilies}
Let $X$ be a reduced adic space locally of finite type over $\Q_p$ and let $\Ncal$ be a family of $(\phi,\Gamma)$-modules over the relative Robba ring $\Rcal_X$.\\
\noindent {\rm (i)} There is a natural open and partially proper subspace $X^{\rm adm}\subset X$ and a family $\Vcal$ of $\mathcal{G}_p$-representations on $X^{\rm adm}$ such that 
\[\Ncal|_{X^{\rm adm}}={\bf D}_{\rig}^\dagger(\Vcal).\]
\noindent {\rm (ii)} The formation $(X,\Ncal)\mapsto (X^{\rm adm},\Vcal)$ is compatible with base change in $X$ and $X=X^{\rm adm}$ whenever the family $\Ncal$ comes from a family of Galois representations. 
\end{theo}

We further will make use of families of \emph{pseudo-characters}, see \cite[1]{BellaicheChenevier} and the references cited there for example. For an adic space $X$ (locally of finite type over $\Q_p$) and a (continuous) map 
$T:\mathcal{G}_p\rightarrow \Ocal_X$ such that $T(xy)=T(yx)$ for all $x,y\in \mathcal{G}_p$ and $\sigma\in\mathfrak{S}_m$, we define $T^\sigma$ as follows: If $\sigma=(j_1,\dots, j_m)\in\mathfrak{S}_m$ is a cycle and $x=(x_1,\dots,x_m)\in A^m$, then $T^\sigma(x)=T(x_{j_1}\cdots x_{j_m})$  and if $\sigma=\sigma_1\cdots\sigma_r$ is a product of cycles, then $T^\sigma=\prod T^{\sigma_i}$. Further we set $S_m(T)=\sum_{\sigma\in \mathfrak{S}_m}\varepsilon(\sigma)T^\sigma$. Then the map $T$ is called a pseudo-character if there exists an integer $d$ such that $S_{d+1}(T)=0$. The smallest integer $d$ such that this is the case is called the dimension of the pseudo-character $T$.

The theory of pseudo-characters of dimension $d$ is well behaved if $d!$ is invertible. As we need to reduce our pseudo-characters and representations modulo $p$ we always assume that $p>d$. The following proposition summarizes the facts on pseudo-characters that will be needed in the sequel.
 \begin{prop} Let $G$ be a group and $A$ a ring.\\
 \noindent {\rm (i)} Let $\rho: G\rightarrow \GL_d(A)$ be a representation. Then ${\rm tr}\,\rho$ is a pseudo-character. \\
 \noindent {\rm (ii)} If $A$ is an algebraically closed field such that $d!\in A^\times$, then $\rho\mapsto {\rm tr}\,\rho$ induces a bijection between semi-simple representations on $d$-dimensional $A$-vector spaces and $A$-valued pseudo-chararacters of dimension $d$. If $G$ is pro-finite and $A=\bar\Fbb_p$ or $A=\bar\Q_p$, then $\rho$ is continuous if and only if ${\rm tr}\,\rho$ is continuous. \\
 \noindent {\rm (iii)} Let $\Fbb$ be a finite field of characteristic $p$ and $G$ a pro-finite group satisfying Mazur's $p$-finiteness condition. If $\bar\tau: G\rightarrow \Fbb$ is a pseudo-character, then the deformation functor of $\bar\tau$  is pro-representable by a complete local Noetherian $W(\Fbb)$-algebra $R_{\bar\tau}$.\\
 \noindent {\rm (iv)} If $\bar\rho:G\rightarrow \GL_d(\Fbb)$ is absolutely irreducible, then $R_{{\rm tr}\,\bar\rho}$ pro-represents the deformation functor of $\bar\rho$. 
 \end{prop}
 
 As $\mathcal{G}_p$ is compact, a family of continuous pseudo-characters $T:\mathcal{G}_p\rightarrow \Ocal_X$ actually factors over $\Ocal_X^+\subset \Ocal_X$. It is then obvious that the reduction modulo $p$, i.e. the composition
\begin{equation}\label{reductionmodp}
\mathcal{G}_p\longrightarrow \Ocal_X^+\longrightarrow \Ocal_X^+/\Ocal_X^{++},\end{equation}
is locally constant on $X$, as it is continuous and $\Ocal_X^+/\Ocal_X^{++}$ is discrete. Here $\Ocal_X^{++}\subset \Ocal_X^+$ denotes the ideal of topologically nilpotent elements. 

If $\bar\tau: \mathcal{G}_p\rightarrow \Fbb$ is a pseudo-character with values in a finite field $\Fbb$, we write $R_{\bar\tau}$ for the universal deformation ring of $\bar\tau$ with universal pseudo-character $\tau$. Further we write $\Xfrak_{\bar\tau}$ for the generic fiber of $\Spf R_{\bar\tau}$. 

For later applications we introduce the following variants: 
We write $\tilde\Dcal_{\bar\tau}^\square$ for the functor 
\[\tilde\Dcal_{\bar\tau}^\square(S)=\left \{(f,\rho)\left|  
\begin{array}{*{20}c}
f:S\rightarrow \Spf\, R_{\bar\tau}\ \text{and a continuous representation}\\   \rho:\mathcal{G}_p\rightarrow \GL_d(\Ocal_S)\ \text{such that}\ {\rm tr}\,\rho=f^\ast \tau
\end{array}
\right.\right\}\]
on the category ${\rm Nil}_{W(\Fbb)}$ of  $W(\Fbb)$-schemes\footnote{Here the morphism $f:S\rightarrow \Spf R_{\bar\tau}$ is a morphism of functors on the category ${\rm Nil}_{W(\Fbb)}$.} $S$ such that $p$ is locally on $S$ nilpotent in the structure sheaf $\Ocal_S$. 
Further we let $\Dcal_{\bar\tau}^\square$ denote the quotient of $\tilde\Dcal_{\bar\tau}^\square$ which is obtained by forgetting the $\mathcal{G}_p$ action on $\Ocal_S^d$. These functors are pro-representable by formal schemes $\tilde\Xcal_{\bar\tau}^\square$ and $\Xcal_{\bar\tau}^\square$ such that $\Xcal_{\bar\tau}^\square\rightarrow \Spf R_{\bar\tau}$ is a $\hat{\GL}_d$-torsor, and $\tilde\Xcal_{\bar\tau}^\square\rightarrow \Xcal_{\bar\tau}^\square$ is formally of finite type. 
We write $\Xfrak_{\bar\tau}^\square$ and $\tilde\Xfrak_{\bar\tau}^\square$ for the generic fibers of these formal schemes. Especially it follows from the above that $\tilde\Xfrak_{\bar\tau}^\square\rightarrow \Xfrak_{\bar\tau}$ is of finite type.

\subsection{Families of trianguline $(\phi,\Gamma)$-modules}
In \cite{Chenevier}, Chenevier has defined the notion of a family of \emph{trianguline} $(\phi,\Gamma)$-modules generalizing the definition of a trianguline $(\phi,\Gamma)$-module, see \cite{Berger} for a survey. Let us briefly recall his definitions. 

Let $X$ be an adic space locally of finite type and let 
\[\delta:\Q_p^\times \longrightarrow \Gamma(X,\Ocal_X^\times)\]
be a continuous character. Then we write $\Rcal_X(\delta)=\Rcal_Xe_\delta$ for the family of rank one $(\phi,\Gamma)$-modules over $X$, on which $\Gamma$ acts via $\delta|_{\Z_p^\times}$ and $\Phi$ acts via $\delta(p)$, i.e. $\Phi(e_\delta)=\delta(p)e_\delta$ and $\gamma.e_\delta=\delta(\gamma)e_\delta$.
We write $\Tcal$ for the space of all continuous characters of $\Q_p^\times$, i.e. 
\begin{equation}\label{Qpcharacters}
\Tcal(A)=\{\text{continuous characters}\ \Q_p^\times\longrightarrow A^\times\}
\end{equation}
for an affinoid algebra $A$. This space is a product $\Tcal=\Wcal\times \Gbb_m$, where $\Wcal$ is the space of continuous characters of $\Z_p^\times$.

Further $\Tcal^{\rm reg}\subset \Tcal$ denotes the Zariski-open subspace of \emph{regular} characters, that is those characters $\delta$ such that $\delta\neq x^{-i}$ and $\delta\neq \chi x^i$ for $i\geq 0$.
Here $x\in\Tcal(\Q_p)$ is the tautological character and $\chi\in\Tcal(\Q_p)$ is the cyclotomic character. 

Fix an integer $d>0$. We write $\Tcal_d^{\rm reg}\subset \Tcal^d$ for the Zariski-open subspace of all $(\delta_1,\dots,\delta_d)\in\Tcal^d$ such that $\delta_i/\delta_j\in\Tcal^{\rm reg}$ for $i\neq j$.

We denote by $\Scal_d^\square$ the space of \emph{trianguline regular rigidified} $(\phi,\Gamma)$-modules, i.e 
\[\Scal_d^\square(X)=\{(D,\Fil_\bullet (D),\delta,\nu=(\nu_1,\dots,\nu_d))\},\]
where $D$ is a $(\phi,\Gamma)$-module over $\Rcal_X$ and $\Fil_\bullet(D)$ is a filtration of $D$ by $\Rcal_X$-subspaces stable under $\Phi$ and $\Gamma$ and which are locally on $X$ direct summands of $D$. Further $\delta=(\delta_1,\dots,\delta_d)$ is an $X$-valued point of $\Tcal_d^{\rm reg}$ and $\nu$ is a family of isomorphisms of $(\phi,\Gamma)$-modules
\[\nu_i:\Fil_{i+1}(D)/\Fil_i(D)\longrightarrow \Rcal_X(\delta_i).\]
This space comes along with a canonical morphism
\[\Scal_d^\square\longrightarrow \Tcal_d^{\rm reg}\]which is smooth of relative dimension $d(d-1)/2$, see \cite[Theorem 3.3]{Chenevier}.

By the construction in \cite{Chenevier} one obtains $\Scal_d^\square$ inductively as follows: We have $\Scal_1^\square=\Tcal$. For $d>1$ consider the open subset  $U_d\subset \Tcal\times \Scal_{d-1}^\square$ mapping to $\Tcal_d^{\rm reg}\subset \Tcal\times\Tcal_{d-1}^{\rm reg}$ under the canonical projection. Then $\Scal_{d}^\square$ is the space attached to a vector bundle $\Mcal_d$ of rank $d-1$ on $U_d$. 

If we inductively apply the same construction with the projective bundle $\mathbb{P}(\Mcal_d)$ attached to $\Mcal_d$, then one ends up with a space $\Scal_d^{\rm ns}$ whose $X$-valued points are given as follows:
\[\Scal_d^{\rm ns} (X)=\{(D,{\rm Fil}_\bullet,\delta)\},\]
where $D$, ${\rm Fil}_\bullet$ and $\delta$ are as above, the subquotient $\Fil_{i+1}(D)/\Fil_i(D)$ is isomorphic to $\Rcal_X(\delta_i)$ (but this time we do not fix the isomorphism) and for all (geometric) points $x\rightarrow X$ the sequences 
\[0\rightarrow \Fil_i(D)\otimes k(x)\rightarrow \Fil_{i+1}(D)\otimes k(x)\rightarrow \big(\Fil_{i+1}(D)/\Fil_i(D)\big)\otimes k(x)\rightarrow 0\]  
are non split for $i=0,\dots,d-1 $. 
By construction, the morphism $\Scal_d^{\rm ns}\rightarrow \Tcal_d^{\rm reg}$ is proper and smooth of relative dimension $(d-1)(d-2)/2$. We will write $\Dcal_d^{\rm ns}$ for the universal family of non-split trianguline $(\phi,\Gamma)$-modules on $\Scal_d^{\rm ns}$

\section{Construction of the local component}
We continue to use the notations of the preceding section. 

By Theorem $\ref{maintheofamilies}$ there is an open subspace $\Scal_d^{\rm ns, adm}\subset \Scal_d^{\rm ns}$ over which the universal trianguline family $\Dcal_d^{\rm ns}$ is induced by a family of $\mathcal{G}_p$-representations.
Let us write $T_\Scal$ for the associated pseudo-character, i.e. the trace of the family of $\mathcal{G}_p$-representations. For an open affinoid $\Spa(A,A^+)\subset \Scal_d^{\rm ns, adm}$ this is a map
 \[T_A:\mathcal{G}_p\longrightarrow A.\] 
Let $\bar\tau:\mathcal{G}_p\rightarrow \Fbb$ be a pseudo-character with values in a finite field $\Fbb$, write $R_{\bar\tau}$ for the universal deformation ring of $\bar\tau$ and $\Xfrak_{\bar\tau}$ for the generic fiber of $\Spf R_{\bar\tau}$ with universal pseudo-character $\tau$. If we write $\Scal^{\rm ns}(\bar\tau)$ for the open and closed subset of $\Scal_d^{\rm ns, adm}$, where $T_\Scal$ reduces to $\bar\tau$ in the special fiber, then there is a morphism
\[f_{\bar\tau}:\Scal^{\rm ns}(\bar\tau)\longrightarrow \Xfrak_{\bar\tau} \]
such that $T_\Scal|_{\Scal^{\rm ns}(\bar\tau)}=f_{\bar\tau}^\ast \tau$. 

We now obtain a map
\begin{equation}\label{defofpitau}
\pi_{\bar\tau}:\Scal^{\rm ns}(\bar\tau)\longrightarrow \Xfrak_{\bar\tau}\times \Tcal_d^{\rm reg},\end{equation}
where the composition of $\pi_{\bar\tau}$ with the projection to $\Tcal_d^{\rm reg}$ is given by restriction of the canonical map $\Scal_d^{\rm ns}\rightarrow \Tcal_d^{\rm reg}$. The following theorem is the main result of this section. 
\begin{theo}\label{properness}
The map
\[\pi_{\bar\tau}:\Scal^{\rm ns}(\bar\tau)\longrightarrow \Xfrak_{\bar\tau}\times \Tcal_d^{\rm reg}\]
is finite and injective.
\end{theo}

\subsection{Partially proper morphisms and the valuative criterion}
Before we continue, we need some preparation about proper morphisms. We collect some definitions and properties from \cite[1.3]{Huber}.
\begin{defn}
Let $f:X\rightarrow Y$ be a morphism of adic spaces.\\
\noindent (i) The morphism $f$ is called \emph{proper} if it is of $ ^+$-weakly finite type and universally closed.\\
\noindent (ii) The morhpism is called \emph{partially proper} if it is locally of $^+$-weakly finite type and universally specializing.
\end{defn}
\begin{rem}
We are not going to define the finiteness condition being locally of \emph{$^+$-weakly finite type} here. As we work in the category of adic spaces locally of fintie type over $\Q_p$, all morphisms are locally of $^+$-weakly finite type. Further all morphisms of finite type (that is in this case all quasi-compact morphisms) are of $^+$-weakly finite type. See \cite[1.2]{Huber} for the precise definitions. \end{rem}
Recall that a \emph{valuation ring} of an adic space $X$ is a tuple $(A,x)$, with $x\in X$ and $A\subset k(x)^+$ open and integrally closed. A \emph{center} of $(A,x)$ is a point $y\in X$ which is a specialization of $x$ such that $k(y)^+=i^{-1}(A)$ under the canonical map $i:k(y)\rightarrow k(x)$.
\begin{prop} Let $f:X\rightarrow Y$ be a morphism of adic spaces which is locally of $^+$-weakly finite type. \\
\noindent {\rm (i)} The morphism $f$ is proper if and only it is partially proper and quasi-compact.\\
\noindent {\rm (ii)} The morphism $f$ is separated if and only if it is quasi-separated and if for all valuation rings $(A,x)$ of $X$ and every center $y$ of $(A\cap k(f(x)),f(x))$ in $Y$ there is at most one center $z$ of $(A,x)$ mapping to $y$.\\
\noindent {\rm (iii)} The morphism $f$ is partially proper if and only if it is quasi-separated and if for all valuation rings $(A,x)$ of $X$ and every center $y$ of $(A\cap k(f(x)),f(x))$ in $Y$ there is a unique center $z$ of $(A,x)$ mapping to $y$.
\end{prop}
\begin{proof}
This is \cite[Lemma 1.3.4, Proposition 1.3.7 and Corollary 1.3.9]{Huber}.
\end{proof}
\begin{prop}
Let $f:X\rightarrow Y$ be a proper and quasi-finite morphism of adic spaces locally of finite type over $\Q_p$. Then $f$ is finite.
\end{prop}
\begin{proof}
This is a special case of \cite[Proposition 1.5.5]{Huber}. 
\end{proof}
We end this subsection by recalling that the inclusion of the admissible locus in a family of $(\phi,\Gamma)$-modules is partially proper.
\begin{prop}\label{Xadmpartproper}
Let $X$ be a reduced adic space locally of finite type over $\Q_p$ and $\Ncal$ be a family of $(\phi,\Gamma)$-modules over $\Rcal_X$. Then the inclusion of the admissible locus 
$X^{\rm adm}\hookrightarrow X$ is partially proper.
\end{prop}
\begin{proof}
This is contained in the statement of Theorem $\ref{maintheofamilies}$.
\end{proof}
\subsection{Properness of the trianguline family}
We are now going to proof Theorem $\ref{properness}$. By the above discussion we need to show that $\pi_{\bar\tau}$ is partially proper, quasi-compact and finite over rigid analytic points.

\begin{lem}\label{quasicompact}
Let $X$ be an adic space of finite type over $\Q_p$ and let $\Ncal_1,\,\Ncal_2$ be families of $(\phi,\Gamma)$-modules of rank $d$ over $\Rcal_X$ such that $\Ncal_1$ is \'etale. Then the subset 
\[Y=\{x\in X\mid \Ncal_1\otimes k(x)\cong \Ncal_2\otimes k(x)\ \text{as}\ (\phi,\Gamma)\text{-modules}\}\]
is quasi-compact.
\end{lem}
\begin{proof}
As $\Ncal_1$ is known to be \'etale, the restriction of $\Ncal_2$ to $Y$ has to be \'etale as well. Hence it suffices to consider $\Ncal_i|_{X\times\boldB_{[r,s]}}$ for some $r<s<1$, as $\Ncal_i$ may be reconstructed from this restriction, compare \cite[Proposition 6.5]{crysfamilies}. Let $\Nfrak_1\subset \Ncal_1|_{X\times\boldB_{[r,s]}}$ be an \'etale lattice and let $\Nfrak_2\subset \Ncal_2|_{X\times\boldB_{[r,s]}}$ be any lattice. Without loss of generality we may assume that $X=\Spa(A,A^+)$ is affinoid. If we set $B=A\langle p^{-s}T,p^rT^{-1}\rangle$, then $X\times\boldB_{[r,s]}=\Spa(B,B^+)$.
Let $\gamma\in\Gamma$ be a topological generator, then $\Phi$ and $\gamma$ act on $\Nfrak_1$ and $\Nfrak_2$ by $p$-isogenies, i.e. $\Phi$ and $\gamma$ induce morphisms 
\[\Nfrak_i\longrightarrow p^{-N}\Nfrak_i\]
for some $N\gg 0$ depending on $\Phi$ and $\gamma$. For $M>0$ we write
\begin{equation}\label{bigcupX_M}
X_M=\{f\in \underline{\rm Hom}(\Nfrak_1,\Nfrak_2)\mid p^M\Nfrak_2\subset f(\Nfrak_1)\subset p^{-M}\Nfrak_2\}.
\end{equation}
This space is of finite type over $X\times \boldB_{[r,s]}$ and in fact $X_M=\Spa(B_M,B_M^+)$ for the ring \[B_M=B\langle p^{M} u_1,\dots,p^Mu_{d^2}, p^{M}u_1^{-1},\dots, p^Mu_{d^2}^{-1}, p^M\det, p^M\det\nolimits^{-1}  \rangle.\]. Further we have
\[\bigcup\nolimits_{M>0}X_M=\underline{\rm Isom}(\Ncal_1|_{X\times\boldB_{[r,s]}},\Ncal_2|_{X\times\boldB_{[r,s]}})=:X_{\infty}\]
and $X_{\infty}$ is the analytification of the scheme $\Spec B\times \GL_d$.  Let $Z\subset \Spec B\times \GL_d$ denote the Zariski-closed subspace defined by the morphisms which commute with $\Phi$ and $\gamma$. The analytification $Z^{\rm ad}$ of $Z$ is a Zariski closed subset of $X_{\infty}$ given by all isomorphisms commuting with $\Phi$ and $\gamma$. 

It follows from $(\ref{bigcupX_M})$ that the images of the morphisms
\[\Spec B_M\longrightarrow \Spec B\times \GL_d\]
cover $\Spec B\times\GL_d$ and hence there is some $M$ such that the image meets each irreducible component of $Z$ (note that $B$ and $B_M$ are noetherian). 

If now $x_0\in\boldB_{[r,s]}$ is any points with residue field not finite over $\Q_p$ (i.e. if  a function $f$ vanishes at $x_0$, then it is identically $0$), then it follows that \[(Z^{\rm ad}\cap X_M)_{x_0}=(Z^{\rm ad}\cap X_M)\times_{\boldB_{[r,s]}}\{x_0\}\] maps onto $Y$ under the canonical projection to $X$. The claim follows, as $Z^{\rm ad}\cap X_M$, and hence $(Z^{\rm ad}\cap X_M)_{x_0}$, are quasi-compact.
\end{proof}
\begin{cor}
The map $\pi_{\bar\tau}$ defined in $(\ref{defofpitau})$ is quasi-compact. 
\end{cor}
\begin{proof}
Recall the definition of the space $\tilde\Xfrak_{\bar\tau}^\square$ from the previous section. The universal family of $\mathcal{G}_p$-representations on $\tilde\Xfrak_{\bar\tau}^\square$ is denoted by $\Vcal_{\bar\tau}^\square$. We consider the fiber product
\[\begin{xy}
\xymatrix{
\tilde\Scal^{\rm ns}(\bar\tau)\ar[r]\ar[d]_{\tilde\pi_{\bar\tau}} & \Scal^{\rm ns}(\bar\tau)\ar[d]^{\pi_{\bar\tau}}\\
\tilde\Xfrak_{\bar\tau}^\square\times\Tcal_d^{\rm reg}\ar[r] &\Xfrak_{\bar\tau}\times\Tcal_d^{\rm reg}.
}
\end{xy}\]
As $\tilde\Xfrak_{\bar\tau}^\square\rightarrow \Xfrak_{\bar\tau}$ is quasi-compact and surjective, it is enough to show that the base change $\tilde\pi_{\bar\tau}$ is quasi-compact. 

Let $U=\Spa(A,A^+)\subset \tilde\Xfrak_{\bar\tau}^{\square}\times\Tcal_d^{\rm reg}$ be affinoid open. We need to show that \[V=(\tilde\pi_{\bar\tau})^{-1}(U)\subset \tilde\Scal^{\rm ns}(\bar\tau) \] is quasi-compact. 

By construction, the morphism $\Scal_d^{\rm ns}\rightarrow \Tcal_d^{\rm reg}$ is proper and hence quasi-compact. It follows that 
\[U'=U\times_{\Tcal_d^{\rm reg}}\Scal_d^{\rm ns}\]
is quasi-compact. Now there are two families $\Ncal_1$ and $\Ncal_2$ of $(\phi,\Gamma)$-modules on $U'$: 
Write $f$ for the composition
\[f:U'\longrightarrow U\longrightarrow \Xfrak_{\bar\tau}^\square\]
an let $g$ denote the projection $g:U'\rightarrow \Scal_d^{\rm ns}$. We set
\[\Ncal_1={\bf D}_{\rm rig}^\dagger(f^\ast \Vcal_{\bar\tau}^\square)\ \text{and}\ \Ncal_2=g^\ast \Dcal_d^{\rm ns}.\]
It follows from Lemma $\ref{quasicompact}$ that 
\[Y=\{y\in U'\mid \Ncal_1\otimes k(y)\cong \Ncal_2\otimes k(y)\ \text{as}\ (\phi,\Gamma)\text{-modules}\}\]
is quasi-compact. On the other hand \[Y\subset U\times_{\Tcal_d^{\rm reg}}\Scal_d^{\rm ns, adm}\] and it follows from the construction that $Y$ maps onto 
\[(\tilde\pi_{\bar\tau})^{-1}(U)\subset \tilde\Scal^{\rm ns}(\bar\tau)=\tilde\Xfrak_{\bar\tau}^\square\times_{\Xfrak_{\bar\tau}}\Scal^{\rm ns}(\bar\tau).\]
\end{proof}
\begin{prop}
The morphism $\pi_{\bar\tau}:\Scal^{\rm ns}(\bar\tau)\rightarrow \Xfrak_{\bar\tau}\times \Tcal_d^{\rm reg}$ is injective on the level of rigid analytic points.
\end{prop}
\begin{proof}
Let $x=(\tau,\delta_1,\dots, \delta_d)\in {\rm im}\, \pi_{\bar\tau}$ be a rigid analytic point and let $\rho:\mathcal{G}_p\rightarrow \GL(V)$ be the associated representation on a $d$-dimensional vector space $V$ over a finite extension $E$ of $\Q_p$. Write $D={\bf D}_{\rig}^\dagger(V)$ for the $(\phi,\Gamma)$-module over $\Rcal_E=\Rcal_{\Spa(E,\Ocal_E)}$ associated with $V$.  We need to show that $D$ has a unique triangulation with weights $(\delta_1,\dots,\delta_d)$. The existence of such a triangulation is obvious. Proceeding by descending induction it is enough to show that 
\[D_1=\left \{x\in X\left |
\begin{array}{*{20}c}
 D/\Rcal_E x\ \text{is torsion free over}\ \Rcal_E,\\ 
 \Phi(x)=\delta_1(p),\ \gamma.x=\delta_1(\gamma)x
\end{array}
\right.\right\}\]
has $E$-dimension $1$. However, this is an obvious consequence of the fact that the filtration has to be nowhere split and $(\delta_1,\dots,\delta_d)\in \Tcal_d^{\rm reg}$.
\end{proof}

\begin{proof}[Proof of Theorem $\ref{properness}$]
We need to show that $\pi_{\bar\tau}$ is proper and injective at rigid analytic points. As we have already shown that $\pi_{\bar\tau}$ is quasi-compact and injective, it remains to show that it is partially proper.
To check that $\pi_{\bar\tau}$ is partially proper, we use the valuative criterion. It is clear that $\pi_{\bar\tau}$ is separated (and especially quasi-separated), as the morphisms $\Scal_d^{\rm ns}\rightarrow \Tcal_d^{\rm reg}$ and $\Xfrak_{\bar\tau}\times\Tcal_d^{\rm reg}\rightarrow \Tcal_d^{\rm reg}$ are separated.

Let $(A,x)$ be a valuation ring of $\Scal^{\rm ns}(\bar\tau)$ and let $y$ be a center of its image in $\Xfrak_{\bar\tau}\times\Tcal_d^{\rm reg}$. We need to show that there is a unique center $z$ of $(x,A)$ mapping to $y$.
Let $y_0\in \Tcal_d^{\rm reg}$ denote the image of $y$. As $\Scal_d^{\rm ns}\rightarrow \Tcal_d^{\rm reg}$ is proper there is a unique center $z\in \Scal^{\rm ns}_d$ of $(x,A)$ mapping to $z$. As the inclusion of the admissible locus is partially proper by Proposition $\ref{Xadmpartproper}$, it follows that $z\in\Scal_d^{\rm ns, adm}$. It is then obvious that $z\in\Scal^{\rm ns}(\bar\tau)\subset \Scal_d^{\rm ns,adm}$.\\
Finally $z$ has to map to $y$ as $\Xfrak_{\bar\tau}\times \Tcal_d^{\rm reg}\rightarrow \Tcal_d^{\rm reg}$ is separated.
\end{proof}
\begin{prop}
The image of a proper map $f: X\rightarrow Y$ of adic spaces locally of finite type over $\Q_p$ is Zariski closed.
\end{prop}
\begin{proof}
It follows from \cite{Kiehl} that $f_{\ast}\Ocal_X$ is a coherent $\Ocal_Y$-module. Hence its support is Zariski-closed. On the other hand ${\rm supp}(f_{\ast}\Ocal_X)$ is the closure of $f(X)\subset Y$. The claim follows as a proper map is closed.
\end{proof}
\begin{cor}
The image of $\pi_{\bar\tau}$ is a Zariski-closed subspace \[X(\bar\tau)^{\rm  reg}\subset \Xfrak_{\bar \tau}\times \mathcal{T}^{\rm reg}.\]
\end{cor}
\begin{proof}
This is an immediate consequence of the above. 
\end{proof}
\begin{defn}
The \emph{local finite slope space} $X(\bar\tau)$ is the (Zariski-) closure of \[X(\bar\tau)^{\rm reg}\subset \Xfrak_{\bar\tau}\times\Tcal^d.\] 
We further write $T_{\bar\tau}:\mathcal{G}_p\rightarrow\Ocal_{X(\bar\tau)}$ for the universal pseudo-character on $X(\bar\tau)$.  If $x\in X(\bar\tau)(\bar\Q_p)$, we write $\rho_x$ for the $\bar\Q_p$-valued $\mathcal{G}_p$-representation with trace 
\[T_{\bar\tau}\otimes k(x):\mathcal{G}_p\longrightarrow \Gamma(X(\bar\tau),\Ocal_{X(\bar\tau)})\longrightarrow k(x). \]
We also write $\pi_{\bar\tau}:\Scal^{\rm ns}(\bar\tau)\rightarrow X(\bar\tau)^{\rm reg}$ for the factorization of the morhphism defined in $(\ref{defofpitau})$. Further we write $\Scal(\bar\tau)$ for the normalization of $X(\bar\tau)$ in $\Scal^{\rm ns}(\bar\tau)$. This yields a space
\[\Scal(\bar\tau)\longrightarrow X(\bar\tau)\]
whose structure morhpism will again be denoted by $\pi_{\bar\tau}$. Finally let $\Vcal_{\bar\tau}^{\rm reg}$ denote the family of $\mathcal{G}_p$-representations on $\Scal^{\rm ns}(\bar\tau)$. The trace of this family is obviously given by $\pi_{\bar\tau}^\ast \big(T_{\bar\tau}|_{X(\bar\tau)^{\rm reg}}\big)$.
\end{defn}
\begin{prop}
The local finite slope space is equidimensional of dimenion 
\[\dim X(\bar\tau)= 1+\tfrac{d(d+1)}{2}=1+\dim B,\]
where $B\subset \GL_d$ is a Borel subgroup. The connected components of $X(\bar\tau)$ are irreducible.
\end{prop}
\begin{proof}
It is clear that $X(\bar\tau)$ has the same dimension as $\Scal_d^{\rm ns}$. This space is smooth of relative dimension $(d-1)(d-2)/2$ over $\Tcal_d^{\rm reg}$. It follows that 
\[\dim \Scal_d^{\rm ns}=\dim \Tcal_d^{\rm reg}+\tfrac{(d-1)(d-2)}{2}=2d+\tfrac{(d-1)(d-2)}{2}=1+\tfrac{d(d+1)}{2}.\]
Further the connected components of $\Scal_d^{\rm ns, adm}$ are irreducible and hence the same is true for the finite slope space. 
\end{proof}
It seems to be very hard to prove anything about the global geometry of $X(\bar\tau)^{\rm reg}$. It is a direct consequence of the construction that the connected components are irreducible. However nothing is implied on the number of connected components. 

\noindent {\bf Question:} Is it true that $X(\bar\tau)$ has finitely many connected components?

 \subsection{The weight map}
Recall that $\Tcal=\Wcal\times \Gbb_m$, where $\Wcal$ is the \emph{weight space}, i.e. the space of all continuous characters of $\Z_p^\times$. For an affinoid algebra $A$ we have 
\[\Wcal(\Spa(A,A^+))=\{\text{continuous homomorphisms}\ \Z_p^\times \longrightarrow A^\times\}.\]
This is the generic fiber of $\Spf( \Z_p[[\Z_p^\times]])$ and decomposes into a disjoint sum of $p-1$ copies of the open unit disc $\Ubb$ over $\Q_p$. Let $\gamma$ be a topological generator of $\Gamma=\Z_p^\times$. For $\delta\in \Wcal (A)$ we define its weight
\begin{equation}\label{weight}
\omega (\delta)=-\left(\frac{\partial\delta}{\partial\gamma}\right)_{\gamma=1}=-\frac{\log(\delta(1+p))}{\log(1+p)}\in A,
\end{equation}
see \cite[2.3.3]{BellaicheChenevier} for example. As the logarithm is given by the power-series 
\[T\longmapsto \sum\nolimits_{n\geq 1} (-1)^{n+1}\frac{T^n}{n}\in\Q_p[[T]]\]
and as this power series defines a homomorphism $\Q_p[T]\rightarrow \Gamma(\Ubb,\Ocal_\Ubb)$
we obtain a map $\omega:\Wcal\longrightarrow \Abb_{\Q_p}^1$ to which we refer as the weight map. 
\begin{lem}\label{Senpoly}
Let $x=(x_0,\delta_1,\dots,\delta_d)\in X(\bar\tau)\subset \Xfrak_{\bar\tau}\times \Tcal^d$, then the image $(\omega(\delta_1),\dots,\omega(\delta_d))\in k(x)^d$ of $x$ under the map
\[\underline{\omega}=(\omega_1,\dots.\omega_d):X(\bar\tau)\longrightarrow \Xfrak_{\bar\tau}\times \Wcal^d\times \Gbb_m^d\longrightarrow \Wcal^d\longrightarrow \Abb_{\Q_p}^d\]
is given by the generalized Hodge-Tate weights of $\rho_x$.
\end{lem}
\begin{proof}
It follows from \cite[Proposition 2.3.3]{BellaicheChenevier} that 
\[\prod\nolimits _{i=1}^d (T-\omega(\delta_i))\in k(x)[T]\]
is the Sen-polynomial of $\Vcal_{\bar\tau}\otimes k(x')$ for any $x'\in \Scal(\bar\tau)$ mapping to $x$. The claim follows from this. See also \cite[Theorem 3.11]{Berger}.
\end{proof}
The following variant of $X(\bar\tau)$ will be of importance in our further discussion.
\begin{defn}
Let \[Z=\bigcup\nolimits_{i}\{t_i=0\}\subset \Abb_{\Q_p}^d=\big(\Spec(\Q_p[t_1,\dots,t_d])\big)^{\rm ad}.\]
Then define $X^0(\bar\tau)=\underline{\omega}^{-1}(Z)\subset X(\bar\tau)$.
 This is the Zariski closed  subspace of $X(\bar\tau)$ consisting of those representations which have at least one Hodge-Tate weight equal to $0$.
\end{defn}

\subsection{Density of crystalline representations}
In the theory of eigenvarieties for definite unitary groups one demands that there is a Zariski-dense accumulation subset of the eigenvariety, arising from automorphic representations with minimal level at $p$. The local analogue for finite slope spaces of this fact it the density of crystalline representations. This result is a variant of Chenevier's density result \cite{Chenevier}. 

Recall that a Zariski-dense subset $Z$ of an adic space $X$ (locally of finite type over $\Q_p$) is called \emph{accumulation}, if for all $z\in Z$ and any open and  connected $U\subset X$ with $z\in U$, the intersection $Z\cap U$ is Zariski-dense in $U$. 

Further we say that a rigid analytic point $x\in X(\bar\tau)(\bar\Q_p)$ is crystalline, if the representation $\rho_x$  is crystalline\footnote{ i.e. the representation associated to the evaluation $T_{\bar\tau}\otimes k(x)$}.
\begin{prop}
The set of crystalline representations is Zariski-dense and accumulation in the connected components of $X(\bar\tau)$ that contain at least one crystalline point. The same is true for $X^0(\bar\tau)$. 
\end{prop}
\begin{proof}
It follows from \cite[Theorem 3.14]{Chenevier} that the crystalline points are Zariski-dense and accumulation in $\Scal_d^{\rm ns}$. Hence the crystalline points are Zariski-dense and accumulation in every component of $\Scal^{\rm ns}(\bar\tau)$ containing at least one crystalline point. As $\Scal^{\rm ns}(\bar\tau)\rightarrow X(\bar\tau)^{\rm reg}$ is a homoemorphism, the crystalline points are Zariski-dense and accumulation in the corresponding subset of  $X^{\rm reg}(\bar\tau)$. 
The claim follows from this, as any neighborhood of a crystalline point in $X(\bar\tau)\backslash X^{\rm reg}(\bar\tau)$ contains a Zariski-dense set of regular crystalline points.
\end{proof}
\begin{defn}
We denote by $\tilde X(\bar\tau)$ the union of all components of $X(\bar\tau)$ which contain a crystalline point. Similarly we denote by $\tilde\Scal(\bar\tau)$ the union of connected components of $\Scal(\bar\tau)$ containing a crystalline point. 
\end{defn}
We recall from \cite[Definition 4.2.3]{BellaicheChenevier} the notion of a refined family of Galois representations. 
\begin{defn} \label{defnrefinedfamily}
A \emph{family of refined $p$-adic representations} of dimension $d$ consists of a reduced rigid analytic space $X$ (or a reduced adic space locally of finite type over $\Q_p$) together with a continuous pseudo-character $T:\mathcal{G}_p\rightarrow \Ocal_X$ and analytic functions $\kappa_1,\dots,\kappa_d, F_1,\dots,F_d\in \Gamma(X,\Ocal_X)$ and a Zariski-dense subset $Z\subset X(\bar\Q_p)$ such that
\begin{enumerate}
\item[(i)] For all $x\in X$ the Hodge-Tate-Sen weights of $\rho_x$ are (with multiplicity) $\kappa_1(x),\dots,\kappa_d(x)$.
\item[(ii)] The representations $\rho_z$ are crystalline for all $z\in Z$.
\item[(iii)] If $z\in Z$, then $\kappa_1(z)<\kappa_2(z)<\dots<\kappa_d(z)$.
\item[(iv)] The eigenvalues of the crystalline Frobenius acting on $D_{\rm cris}(\rho_z)$ are distinct and given by $(p^{\kappa_1(z)}F_1(z),\dots, p^{\kappa_d(z)}F_d(z))$.
\item[(v)] For each $n$ there exists a continuous character $\delta: \Z_p^\times \rightarrow \Gamma(X,\Ocal_X^\times)$ whose derivative at $1\in\Z_p^\times$ is the map $\kappa_n$ and whose evaluation at any point $z\in Z$ is the evaluation of the $\kappa_n(z)$-th power.  
\item[(vi)] For any non-negative integer $C$ let $Z_C$ denote the set of $z\in Z$ such that 
\begin{align*}
\kappa_2(z)-\kappa_1(z)&>C\ \text{and}\\
\kappa_{n+1}(z)-\kappa_n(z)&>C(\kappa_n(z)-\kappa_{n-1}(z))\ \text{for all}\ n=2,\dots, d-1.
\end{align*}
Then $Z_C$ accumulates at every point of $Z$. 
\end{enumerate}
\end{defn}

Recall the definition of the weights $\omega_1,\dots,\omega_d$ on $X(\bar\tau)$. We set  \[\Zcal=\{x\in X(\bar\tau)(\bar\Q_p)\mid \rho_x\ \text{crystalline and}\ \omega_1(x)<\dots,<\omega_d(x)\}.\] 
\begin{prop}
The family \[(\tilde X(\bar\tau), T_{\bar\tau}|_{\tilde X(\bar\tau)}, \Zcal, \delta_1(p),\dots,\delta_d(p),\omega_1,\dots,\omega_d)\]
is a family of refined $p$-adic representations. 
\end{prop}
\begin{proof}
It is clear that the sets $\Zcal_C$ accumulate at all points of $\Zcal$ and that $\Zcal$ is Zariski-dense in $\tilde X(\bar\tau)$. Further the weights $\omega_i$ are the derivations of $\delta_i|_{\Z_p^\times}$ at the origin. 
Condition (i) now follows from Lemma $\ref{Senpoly}$ and (iv) follows from \cite[Proposition 2.4.1]{BellaicheChenevier}.
\end{proof}

\begin{cor} \label{universalrefined}
The finite slope space $\tilde X(\bar\tau)$ is the universal refined family with residual pseudo-character $\bar\tau$ in the following sense: Let $(X,T,Z,F_\bullet, \kappa_\bullet)$ be a refined family such that $T$ reduces to $\bar\tau$ modulo $\Ocal_X^{++}$. Then there exists unique morphism 
$f:X\rightarrow \tilde X(\bar\tau)$ such that $T=f^\ast T_{\bar\tau}$ and $Z\subset f^{-1}(\Zcal)$ and such that $\kappa_n=f^\ast \omega_n$ and $F_n=f^\ast\delta_n(p)$. 
\end{cor}
\begin{proof}
By assumption $\kappa_n$ is the derivation at the origin of a character
 \[\epsilon_n:\Z_p^\times \longrightarrow \Gamma(X(\bar\tau),\Ocal^\times_{X(\bar\tau)})\]
 such that the evaluation at $z\in Z$ is the $\kappa_n(z)$-th power. As $Z$ is Zariski-dense, this continuous family of characters is unique. These data define a map
 \[f=(T,F_\bullet, \epsilon_\bullet):X\longrightarrow \Xfrak_{\bar\tau}\times \Gbb_m^d\times \Wcal^d=\Xfrak_{\bar\tau}\times \Tcal^d.\]
 As $X$ is reduced it is enough to show that this map factors topologically over $\tilde X(\bar\tau)$. However, this follows from the fact that $\tilde X(\bar\tau)\subset \Xfrak_{\bar\tau}\times \Tcal^d$ is Zariski-closed and contains all crystalline representations with the regularity conditions (iii) and (iv), and the density of crystalline representations in $X$. To prove  $\kappa_n=f^\ast \omega_n$ and $F_n=f^\ast\delta_n(p)$ is is enough to consider their evaluation at the points in $Z$, where the claim is obvious. 
 \end{proof}

\subsection{Comparison with Kisin's finite slope subspace}
In this section we assume that $d=2$ and want to link our construction to Kisin's finite slope subspace, as defined in \cite{KisinsXfs}. As our space is defined as a subspace of $\Xfrak_{\bar\tau}\times\Tcal^d$ while Kisin's is a subspace of $\Xfrak_{\bar\rho}^{\rm vers}\times \Gbb_m$ (where $\Xfrak_{\bar\rho}^{\rm vers}$ is the generic fiber of the versal deformation ring of a residual representation $\bar\rho$), we first have to adapt our definitions to Kisin's setting. Let
\[\bar \rho:\mathcal{G}_p\longrightarrow \GL_2(\Fbb)\]
be a continuous representation of $\mathcal{G}_p$ with coefficients in a finite field $\Fbb$. For simplicity we assume that $\bar\rho$ is not a twist of the trivial representation. 
We write $R_{\bar\rho}^{\rm vers}$ for the versal deformation ring of $\bar\rho$ and $\Xfrak_{\bar\rho}^{\rm vers}$ for the generic fiber of $\Spf R_{\bar\rho}^{\rm vers}$. Denote by $\bar\tau$ the trace of $\bar\rho$. Then we have a morphism
\[\Xfrak_{\bar\rho}^{\rm vers}\longrightarrow \Xfrak_{\bar\tau}\]
and we denote by $X(\bar\rho)\subset \Xfrak_{\bar\rho}^{\rm vers}\times \Tcal_d$ the pullback of $X(\bar\tau)$ along this map. The pullback of $X^0(\bar\tau)\subset X(\bar\tau)$ will we denoted by $X^0(\bar\rho)$.  We write $\Vcal_{\bar\rho}$ for the family of $\mathcal{G}_p$-representations on $X(\bar\rho)$ (or $X^0(\bar\rho)$) given by the versal deformation of $\bar\rho$. 
\begin{prop}The morphism
\[f: X(\bar\rho)\longrightarrow \Xfrak_{\bar\rho}^{\rm vers}\times\Tcal^2\longrightarrow \Xfrak_{\bar\rho}^{\rm vers}\times \Gbb_m\]
induced by the projection $\Tcal^2=\Wcal^2\times\Gbb_m^2\rightarrow \Gbb_m$ to the first factor of the $\Gbb_m^2$, is finite and generically injective. The same is true for the restriction $f|_{X^0(\bar\rho)}$.
\end{prop}
\begin{proof}
It is enough to show that the morphism \[g: X(\bar\rho)\longrightarrow \Xfrak_{\bar\rho}^{\rm vers}\times \Gbb_m^2\]
has the required properties, as we can recover the second eigenvalue of $\Phi$ by $(\rho,\lambda)\mapsto (\rho,\lambda,\lambda^{-1}\delta_{\det\rho}(p))$, where 
\begin{equation}\label{deltadet}
{\bf D}_{\rig}^\dagger(\det \rho)=\Rcal_{\Xfrak_{\bar\rho}^{\rm vers}}(\delta_{\det \rho}).
\end{equation}
Now, given an affinoid open subspace $U\subset \Xfrak_{\bar\rho}^{\rm vers}\times \Gbb_m^2$ there is a unique $\Phi$-stable subspace of ${\bf D}^\dagger_{\rig}(\Vcal_{\bar\rho}|_U)$, as the extension of the underlying $\phi$-modules is generically non-split. We then define $\delta_1|_{\Z_p^\times}$ as the unique character such that $\Gamma$ acts on a generator of this subspace via multiplication with $\delta_1$. Then $\delta_2$ is determined by $\delta_1\delta_2=\delta_{\det\rho}$ and $(\ref{deltadet})$. It follows that the map $g$ is quasi-compact and hence proper, as it it obviously partially proper.
It is clear that away from the diagonal of $\Gbb_m^2$ the preimage of a rigid analytic point consists just of one point. However, as $\bar\rho$ is not a direct sum of two equal characters, so is every representation $\rho\in \Xfrak_{\bar\rho}^{\rm vers}$ and hence the pre-image of the points on the diagonal of $\Gbb_m^2$ consists of at most $2$ points. 
\end{proof}

For the rest of this subesction, we write $Y(\bar\rho)$ for the image of $X^0(\bar\rho)$ under $f$ and $\tilde Y(\bar\rho)$ for the  union of the irreducible components $Y(\bar\rho)$ that contain a crystalline point. 
Recall that Kisin \cite[5]{KisinsXfs} has defined a finite slope space $X_{\rm fs}$ as a subspace of the product $\Xfrak_{\bar\rho}^{\rm vers}\times \Gbb_m$. 
This subspace is characterized by the following conditions: Let $X$ be an reduced adic space locally of finite type equipped with a family of $\mathcal{G}_p$-representations $\Vcal$ on the trivial vector bundle of rank $2$ and let $Y\in \Gamma(X,\Ocal_X^\times)$ be an invertible global section on $X$. Then there is a unique closed subspace $X_{\rm fs}\subset X$ such that 
\begin{enumerate}
\item[(i)] The Sen-Polynomial $P(T)$ of the representations $\rho\in X$ factors as $P(T)=TQ(T)$, when restricted to $X_{\rm fs}$.
\item[(ii)] For all integers $j\leq 0$ the subspace $X_{{\rm fs}, Q(j)}=\{x\in X_{\rm fs}\mid Q(j)(x)\neq 0\}$ is scheme-theoretically dense in $X_{\rm fs}$.
\item[(iii)] For any affinoid algebra $A$ and any $Y$-small map $f: \Spa(A,A^+)\rightarrow X$ such that $f$ factors over $X_{Q(j)}$ for every integer $j\leq 0$, the map $f$ factors over $X_{\rm fs}$ if and only if every $A$-linear $\mathcal{G}_p$-invariant period 
\[h_A: \Gamma(\Spa(A,A^+),f^\ast\Vcal)\longrightarrow B^+_{\rm dR}\hat\otimes_{\Q_p} A\]
factors over $(B_{\rm cris}\hat\otimes_{\Q_p}A)^{\phi=f^\ast Y}\subset B_{\rm dR}^+\hat\otimes_{\Q_p}A$.
\end{enumerate}
Here the condition being $Y$-small means that there exists a finite extension $K$ over $\Q_p$ and $\lambda\in A\otimes_{\Q_p}K$ such that $K(\lambda)$ is a product of finite field extensions of $K$ and $Y\lambda^{-1}-1$ is topologically nilpotent in $\Spa(A, A^+)\otimes K$.

 The following result shows that Kisins notion coincides with our definition. 
\begin{prop} With the above notations $\tilde Y(\bar\rho)$ is precisely given by those irreducible components of 
\[\big(\Xfrak_{\bar\rho}^{\rm vers}\times \Gbb_m\big)_{\rm fs}\]
that contain at least one crystalline point. 
\end{prop}
\begin{proof}
First it is clear that \[\big(\Xfrak_{\bar\rho}^{\rm vers}\times \Gbb_m\big)_{\rm fs}\subset Y(\bar\rho).\]
This may be checked on the level of points and follows from the description of the points of Kisin's finite slope space given in \cite[Proposition 10.4]{KisinsXfs}:
The points with generalized Hodge-Tate weights $0$ and $k$, where $k$ is not a non-negative integer, are dense, and it follows from \cite[Proposition 4.3]{Colmez} that the existence of a crystalline period is equivalent to the condition of being trianguline (the regularity condition on the weights of the triangulations is automatic as we demand that the Hodge-Tate weights are distinct). 

Further we check that $\tilde Y(\bar\rho)_{\rm fs}=\tilde Y(\bar\rho)$.  
As in the section above we may consider the maximal subset of crystalline points $Z\subset \tilde Y(\bar\rho)$ which makes $\tilde Y(\bar\rho)$ into a refined family. It then follows from \cite[Theorem 3.3.3]{BellaicheChenevier} that we have $\tilde Y(\bar\rho)_{\rm fs}=\tilde Y(\bar\rho)$. These two statements imply the claim:

Conditions (i) and (ii) are obviously satisfied by construction of $\tilde Y(\bar\rho)$. Let $f:\Spa(A,A^+)\rightarrow \Xfrak_{\bar\rho}^{\rm vers}\times\Gbb_m$ be a map as in (iii), and we further assume that the crystalline points are dense in $\Spa(A,A^+)$. 

If every de Rham period $h_A$ is crystalline, then $f$ factors over Kisin's finite slope space and our first claim shows that $f$ factors over $\tilde Y(\bar\rho)$. 

If on the other hand $f$ factors over $\tilde Y(\bar\rho)$, then the second claim shows that every de Rham period $h_A$ is crystalline (with the desired action of the crystalline Frobenius). 
\end{proof}

\section{The global construction}

Fix a finite set $S$ of primes and let $E$ be a finite extension of $\Q$ such that $p$ completely splits in $E$. Let $E_S\subset \bar\Q$ be the maximal subextension of $E$ in a fixed algebraic closure $\bar\Q$ which is unramified outside the places of $E$ lying over the primes in $S$. Further we write $\mathcal{G}_{E,S}={\rm Gal}(E_S,E)$ for the corresponding Galois-group. 
For the rest of this section we fix a continuous pseudo-character
\[\bar\tau:\mathcal{G}_{E,S}\longrightarrow \Fbb\]
and write $\bar\tau_p$ for the restriction of $\bar\tau$ to $\mathcal{G}_p$ induced by choosing a prime of $E$ above $p$. Again we write $R_{\bar\tau}$ for the universal deformation ring of $\bar\tau$ and $\Xfrak_{\bar\tau}$ for its generic fiber. Consider the map 
\begin{equation}\label{localization}\Xfrak_{\bar\tau}\longrightarrow \Xfrak_{\bar\tau_p}\end{equation}
given by restricting the universal deformation of $\bar\tau$ to $\mathcal{G}_p$. 
\begin{defn}
The \emph{global finite slope space} $X(\bar\tau)$ is the pullback of the local finite slope space $X(\bar\tau_p)$ along the map $(\ref{localization})$. We write $T_{\bar\tau}$ for the universal pseudo-character on $X(\bar\tau)$.
\end{defn}
Recall that a rigid analytic space (or an adic space locally of finite type over $\Q_p$) is called \emph{quasi-Stein} if there is a covering $X=\bigcup_{i\geq 0} U_i$ by open affinoids with $U_i\subset U_{i+1}$ such that $\Gamma(U_{i+1},\Ocal_X)\rightarrow \Gamma(U_{i},\Ocal_X)$ has dense image. Further $X$ is called \emph{nested} if there is a covering $X=\bigcup_{i\geq 0} V_i$ by open affinoids with $V_i\subset V_{i+1}$ such that $\Gamma(V_{i+1},\Ocal_X)\rightarrow \Gamma(V_i,\Ocal_X)$ is compact. 
\begin{prop}
\noindent {\rm (i)} The finite slope space $X(\bar\tau)$ is a nested quasi-Stein space. \\
\noindent {\rm (ii)} The map $X(\bar\tau)\rightarrow \Wcal^d$ is quasi-Stein, i.e. the preimage of an affinoid open is a quasi-Stein space. \\
\noindent {\rm (iii)} If $Z$ is a connected component of the local finite slope space $X(\bar\tau_p)$, then the pre-image of $Z$ under $X(\bar\tau)\rightarrow X(\bar\tau_p)$ has only finitely many irreducible components. 
\end{prop}
\begin{proof}
The first point is a direct consequence of the construction and the second point follows from this. Finally (iii) follows from the fact the the kernel of $R_{\bar\tau_p}\rightarrow R_{\bar\tau}$ is generated by finitely many elements. 
\end{proof}
\begin{lem}\label{triangulineGpfamily} Let $U\subset X(\bar\tau)$ denote the Zariski-open subspace where the restriction $T_{\bar\tau}|_{\mathcal{G}_p}$ is absolutely irreducible. 
Then the projection $U\rightarrow X(\bar\tau_p)$ factors over $\Scal(\bar\tau_p)$.
\end{lem}
\begin{proof}
Let $U_p\subset X(\bar\tau_p)$ denote the open subspace such that the pseudo-character $T_{\bar\tau_p}$ is absolutely irreducible when restricted to $U_p$. If $U_{\Scal}$ denotes its pre-image under $\pi_{\bar\tau_p}$, then the induced map
\[f=\pi_{\bar\tau_p}|_{U_\Scal}:U_{\Scal}\longrightarrow U_p\]
is \'etale, as for absolutely irreducible representations it is equivalent to deform the representation itself, or its trace. 

Let $x\in X(\bar\tau)$ with corresponding representation $\rho_x:\mathcal{G}_{E,S}\rightarrow \GL_d(\bar\Q_p)$. This representation has actually coefficients in $k(x)$, as the images of the Frobenius elements at the unramified places have entries in $k(x)$ and as they generate a dense subgroup of $\mathcal{G}_{E,S}$ by the Chebotarev-density theorem. It follows that the \'etale cover $f'$ in the cartesian diagram
\[\begin{xy}
\xymatrix{
\Scal^{\rm ns}(\bar\tau) \ar[r]\ar[d]_{f'} & U_{\Scal}\ar[d]^f\\
U\ar[r] & U_p
}
\end{xy}\]
splits into a disjoint union of $U$. Especially $f'$ admits a section which implies the claim.  
\end{proof}
\begin{cor}\label{familyofGalrep}
 Assume that $\bar\tau$ is the trace of an absolutely irreducible representation $\bar\rho$.\\
 \noindent {\rm (i)}  There is a family $\Vcal_{\bar\tau}$ of $\mathcal{G}_{E,S}$-representations on $X(\bar\rho)$ with trace $T_{\bar\tau}$.\\
 \noindent {\rm (ii)} Let $Z\subset X(\bar\tau)$ be an irreducible component such that $T_{\bar\tau}|_{\mathcal{G}_p}$ is absolutely irreducible at one point of $Z$ and such that at least one point of $Z$ maps to the regular locus $X^{\rm reg}(\bar\tau_p)\subset X(\bar\tau_p)$. Then the family $\Vcal_{\bar\tau}|_{\mathcal{G}_p}$ is trianguline at all points of $Z$.\\
 \noindent {\rm (iii)} The family $\Vcal_{\bar\tau}|_{\mathcal{G}_p}$ is a trianguline family over an Zariski-open subset of $Z$. 
\end{cor}
\begin{proof}
As $\bar\tau={\rm tr}\,\bar\rho$ for an absolutely irreducible representations $\bar\rho$, the universal deformation ring $R_{\bar\tau}$ of $\bar\tau$ also pro-represents the deformation functor of $\bar\rho$. Hence there is a canonical family $\Vcal_{\bar\tau}$ of $\mathcal{G}_{E,S}$-representations on $\Xfrak_{\bar\tau}$ and hence also on $X(\bar\tau)$.

Let $U$ be the Zariski-open subset of the preceding lemma. By assumption the intersection $U'=U\cap Z$ is non-empty and hence $U'$ is Zariski-dense in $Z$.
It follows from the above lemma and the fact that the pseudo-characters on $U'$ are absolutely irreducible that $\Vcal_{\bar\tau}|_{\mathcal{G}_p}$ is the trianguline family pulled back to $U'$. 
The last point follows from this and the second point follows from the properness of the flag variety: If $z\in Z$ maps to $X(\bar\tau_p)\backslash X^{\rm reg}(\bar\tau_p)$, then we can specialize a trianguline filtration from the regular locus to $z$ and obtain again a trianguline filtration\footnote{However, note that this does not automatically imply that the whole family $\Vcal_{\bar\tau}|_{\mathcal{G}_p}$ is trianguline.}. 
\end{proof}

\subsection{The action of the Hecke algebra}
To define the action of the Hecke-algebra we need to recall the set up of \cite{BellaicheChenevier}.
Fix a set $S_0$ of primes in $\Q$ and write $\Hcal_{\rm ur}$ for the Hecke-algebra
\[\Hcal_{\rm ur}=\Ccal_c^\infty(\GL_d(\hat\Z_{S_0})\backslash\GL_d(\Abb_{S_0})/\GL_d(\hat\Z_{S_0}), \Z).\]
Here, and in the following, $\Ccal_c^\infty(-,M)$ denotes the set of locally constant functions with compact support and values in $M$. This Hecke-algebra will take care of the information at the unramified places. The situation at $p$ is controlled by the so called \emph{Atkin-Lehner} ring which is defined as follows:

Let $T\subset B\subset \GL_d(\Q_p)$ be the diagonal torus resp. the Borel subgroup of upper triangular matrices and let $I\subset \GL_d(\Q_p)$ denote the standard Iwahori subgroup. Further write $T^\circ\subset T$ for the $\Z_p$-valued points of the diagonal torus and 
\begin{align*}
U&=\{{\rm diag}(p^{a_1},\dots,p^{a_d})\in T\mid a_i\in \Z\}\\
U^-&=\{{\rm diag}(p^{a_1},\dots,p^{a_d})\in U\mid a_i\geq a_{i+1}\}.
\end{align*}
Then the Atkin-Lehner ring is the subring 
\[\mathcal{A}_p\subset \Ccal_c^\infty\big(I\backslash\GL_d(\Q_p)/I,\Z\big[\tfrac{1}{p}\big]\big)\] generated by characteristic functions of the subsets $IuI$ with $u\in U^-$ and their inverses, compare \cite[6.4]{BellaicheChenevier} for these definitions. By \cite[Proposition 6.4.1]{BellaicheChenevier} we have $\mathcal{A}_p\cong \Z[U]$.

Finally, the Hecke-algebra we are going to consider is \[\Hcal=\Hcal_{\rm ur}\otimes_\Z \mathcal{A}_p\subset \Ccal_c^\infty(\GL_d(\Abb_{S_0,p}),\Q).\]

Assume that $S\cap S_0=\emptyset$. Then there is an obvious ring-homomorhpism $\psi_{\rm ur}^{\rm gal}:\Hcal_{\rm ur}\rightarrow R_{\bar\tau}$: Let $\ell\in S_0$ and let 
\[P(T_1,\dots, T_d)\in \Z[T_1,\dots,T_d]^{S_d}\]
be an element of the Hecke-algebra at the prime $\ell$. Then $P$ is mapped to its evaluation on the eigenvalues of the Frobenius ${\rm Frob}_\ell$ at $\ell$.

Further we can construct an action of the Atkin-Lehner ring on $\Tcal^d$ as follows. Let 
\[h_i={\rm diag}(1,\dots,1,p,1,\dots,1)\in U,\] with $p$ at the $i$-th entry. There is a ring homomorphism \[\psi_p^{\rm gal}:\mathcal{A}_p\longrightarrow \Gamma(\Tcal^d,\Ocal_{\Tcal^d})\] such that $\psi_p^{\rm gal}(h_i)=\delta_i(p)$.  The above defines a ring homomorphism
\[\psi^{\rm gal}=\psi^{\rm gal}_{\rm ur}\otimes \psi^{\rm gal}_p:\Hcal=\Hcal_{\rm ur}\otimes_\Z\mathcal{A}_p\longrightarrow \Gamma(X(\bar\tau),\Ocal_{X(\bar\tau)}).\]
\begin{prop}
Assume that the set $S_0$ has Dirichlet-density $1$. \\
\noindent {\rm (i)} The homomorphism $\psi^{\rm gal}$ has dense image and \[\psi_{\rm ur}^{\rm gal}(\Hcal_{\rm ur})\subset \Gamma(X(\bar\tau),\Ocal_{X(\bar\tau)}^+).\] 
\noindent {\rm (ii)} Let $u_0,\dots,u_m\in\mathcal{A}_p^\times$. Then 
\[\nu_m=(\omega, \psi(u_0)^{-1},\dots,\psi(u_m)^{-1}):X(\bar\tau)\longrightarrow \Wcal^d\times\Gbb_m^m\]
induces surjections
\[\psi^{\rm gal}\otimes\nu_m^\sharp: \Hcal\otimes_{\Z}\Gamma(V,\Ocal_{\Wcal^d\times\Gbb_m^m})\longrightarrow \Gamma(\nu_m^{-1}(V),\Ocal_{X(\bar\tau)})\]
for all affinoid opens $V\subset \Wcal^d\times\Gbb_m^m$ if and only if $\nu_m$ is finite. 
\end{prop}
\begin{proof}
\noindent (i)  This follows from the Chebotarev desity theorem and the assumption that $S_0$ has density 1.\\
\noindent (ii) If the map is finite, then 
$\Gamma(\nu_m^{-1}(V),\Ocal_{X(\bar\tau)})$
is generated by finitely many sections over $\Gamma(V,\Ocal_{\Wcal^d\times \Gbb_m^m})$. Hence density of the image implies surjectivity.\\
Assume conversely that the map is surjective. 
As $\Gamma(\nu_m^{-1}(V),\Ocal_{X(\bar\tau)})$ is generated by a countable set of generators over $\Gamma(V,\Ocal_{\Wcal^d\times\Gbb_m^m})$, it has to be generated by finitely many elements. 
\end{proof}

\subsection{Comparison with eigenvarieties}
To compare finite slope spaces with eigenvarieties we recall the aximoatic set up from \cite[7.2]{BellaicheChenevier}.
Let $E$ be an \'etale $\Q$-algebra of rank $2$ with non-trivial $\Q$-automorhpism $c$. 
Further let $\Delta$ be a central simple $E$-algebra of rank $d^2$ and let $x\mapsto x^\ast$ be a $\Q$-involution of the second kind on $\Delta$. 
Let $G$ be the reductive group over $\Q$ attached to $\Delta$, i.e.
\[G(A)=\{g\in (\Delta\otimes_k A)^\times \mid gg^\ast=1\}.\]
In the following we will assume that $p$ splits in $E$ and that $G$ splits at $p$, i.e.  $G\otimes_\Q \Q_p=\GL_{d,\Q_p}$. Further we assume that $G$ is compact at infinity. 
We recall the notion of a $p$-refined automorphic representation from \cite[7.2.2]{BellaicheChenevier}.
\begin{defn}
A \emph{$p$-refined automorphic representation of weight $\underline{k}\in\Z^d$ of $G$} is a pair $(\pi,\Rcal)$ consisting of an irreducible automorphic representation $\pi$ of $G(\Abb)$ such that
$\pi_\infty$ is induced by restriction of the representation of $\GL_d(\C)$ of highest weight $\underline{k}=(k_1,\dots,k_d)$ under $G(\R)\hookrightarrow \GL_d(\C)$. Further $\pi_p$ is unramified and $\Rcal$ is a refinement of $\pi_p$, that is, an ordering of the eigenvalues of the Satake parameter of $\pi_p$. In other words $\Rcal$ is given by a character $\chi:U\cong T/T^\circ\rightarrow \mathbb{C}^\times$. We further require that the refinement is \emph{accessible} which means that there is an embedding \[\pi_p\longrightarrow {\rm Ind}_B^{\GL_d(\Q_p)}\chi.\] 
\end{defn}
Let $S_0$ be a set of primes $\ell$ such that $\ell$ splits in $E$ and $G\otimes_\Q \Q_\ell\cong \GL_d(\Q_\ell)$. We further choose a model of $G$ over $\mathbb{Z}$ and assume that $G(\Z_\ell)$ is maximal compact for $\ell\in S_0$. In this case the data of a $p$-refined automorphic representation $(\pi,\Rcal)$ gives rise to a character $\psi=\psi_{(\pi,\Rcal)}$ of the Hecke algebra $\Hcal$ defined in the previous section: The restriction $\psi_{\rm ur}$ of $\psi$ to the unramified Hecke algebra is given by the Satake parameter of $\pi_{S_0}$ while the restriction $\psi_p$ of $\psi$ to $\mathcal{A}_p$ is defined by
\[\psi_p|_U=\chi\delta_B^{-1/2}\delta_{\underline{k}}. \]
Here $\delta_{\underline{k}}$ is the character $(z_i)\mapsto (z_i^{k_i})$ of the diagonal torus and $\delta_B$ is the modulus character. 
This defines a $\mathbb{C}$-valued character of $\Hcal$ which takes values in $\bar\Q$ by \cite[Lemma 6.2.7]{BellaicheChenevier} and hence may be viewed as a $\bar\Q_p$-valued character (after choosing embeddings $\bar\Q\hookrightarrow \mathbb{C}$ and $\bar\Q\hookrightarrow \bar\Q_p$).  
Let $\Zcal$ be a subset of the set 
\begin{equation}\label{refinedautomorphic}
\left\{ 
\begin{array}{*{20}c}
(\psi_{(\pi,\Rcal)},\underline{k})\in {\rm Hom}(\Hcal,\bar\Q_p)\times \Z^d \\
 (\pi,\Rcal)\ \text{is a $p$-refined automorphic representations of weight}\ \underline{k}
\end{array}
\right\}.
\end{equation}
Further let $u_0={\rm diag}(p^{d-1},\dots,p,1)\in U^-\subset \mathcal{A}_p^\times$. 
\begin{defn}\label{defneigenvar}
An \emph{eigenvariety} for $\Zcal$ is a reduced rigid analytic space $X$ equippend with a ring homomorphism $\psi:\Hcal\rightarrow \Gamma(X,\Ocal_X)$, a morhpism $\omega: X\rightarrow \Wcal^d$ and a Zariski-dense accumulation subset $Z\subset X(\bar\Q_p)$ such that
\begin{enumerate}
\item[(i)] The map $\nu=(\omega,\psi(u_0)^{-1}):X\rightarrow \Wcal^d\times \Gbb_m$ is finite.
\item[(ii)] For all open affinoid $V\subset \Wcal^d\times \Gbb_m$ the natural map
\[\psi\otimes\nu^\ast:\Hcal\otimes\Gamma(V,\Ocal_{\Wcal^d\times\Gbb_m}) \longrightarrow \Gamma(\nu^{-1}(V),\Ocal_X)\]
is surjective. 
\item[(iii)] There is a bijection $Z\rightarrow \Zcal$ induced by the evaluation of $\psi$ and $\omega$. 
\end{enumerate}
\end{defn}
If it exists, an eigenvariety is unique up to unique isomorphism, see \cite[Proposition 7.2.8]{BellaicheChenevier}. We have seen above that the global finite slope space has properties that are at least similar to the ones of eigenvarieties. To have a better comparison with eigenvarieties we need to specialize our set up. In the following we will assume that we are in one of the following cases.
\begin{enumerate}
\item[(a)]  $E=\Q\times\Q$ and $\Delta=\Delta_1\times\Delta_1^{\rm opp}$ for a definite quaternion algebra $\Delta_1$ over $\Q$. Especially $d=2$. 
\item[(b)]  $E$ an imaginary quadratic extension and $\Delta$ is associated to the hermitian form $q:(z_1,\dots, z_d)\mapsto \sum N(z_i)$, where $N:E\rightarrow \Q$ is the norm map. 
\item[(c)]  $E$ an imaginary quadratic extension and $G$ is a definite unitary group such that $G(\Q_\ell)$ is either quasi-split or isomorphic to the group of invertible elements of a central division algebra over $\Q_\ell$. The latter case occurring for at least one prime $q$.
\end{enumerate}
In the following we assume that we are in the situation of \cite[7.5]{BellaicheChenevier}. That is we assume that the set $S_0$ has Dirichlet-density $1$ and fix a compact open subgroup $K^p=\prod_{\ell\neq p}K_\ell\subset G(\Abb^p_f)$. Finally we choose a finite set $S$ of primes such that  $p\in S$ and for each $\ell \notin S$ or $\ell\in S_0$ the compact open subgroup $K_\ell$ is maximal hyperspecial or very special compact. Choose the decomposed idempotent \[e\in  \Ccal_c^\infty(G(\Abb^{p,S_0}),\bar\Q)\otimes 1_{\Hcal_{\rm ur}}\subset \Ccal_c^\infty(G(\Abb_f^p),\bar\Q)\] 
which satisfies $ee_{K^p}=e_{K^p}e$. In case (c) we further assume that $e_q$ vanishes on $1$-dimensional representations of the division algebra $G(\Q_q)$.

 Let $\Zcal_e$ be the set of all $p$-refined representations such that $e(\pi_p)\neq 0$ and let $\Zcal\subset \Zcal_e$ be the set of all $p$-refined representations which are \emph{non monodromic principal series} at the primes in a fixed subset $S_N\subset S$ (which is empty in case (c)). Let $(Y,\psi,\omega,Z)$ be the corresponding eigenvariety which exists by \cite[7.5]{BellaicheChenevier}.
 \begin{prop}
 With the above notations, there exists a continuous pseudo-character $T_Y: \mathcal{G}_{E,S}\rightarrow \Gamma(Y,\Ocal_Y)$ such that for all $z=(\pi,\Rcal)\in Z$ we have $T\otimes k(z)={\rm tr}\, \rho_{z}$, where $\rho_z$ is the Galois-representation\footnote{In case (b) we need to assume the folklore conjecture (see \cite[Conjecture 6.8.1]{BellaicheChenevier} for example) on the existence of this representation} associated with the automorphic representation $\pi$. 
 \end{prop}
 \begin{proof}
 This is one of the main results of \cite[7.5]{BellaicheChenevier}.
 \end{proof}
 As a consequence of the above proposition we introduce the following notation: Let $\bar\tau:\mathcal{G}_{E,S}\rightarrow \Fbb$ be a pseudo-character. Then $Y(\bar\rho)\subset Y$ is the open and closed subset of the eigenvariety $Y$, where the reduction modulo $\Ocal_Y^{++}$ of the pseudo-character $T_Y$ is isomorphic to $\bar\tau$.  
\begin{theo}
There is a closed embedding $Y(\bar\tau)\rightarrow X(\bar\tau)$ of the $\bar\tau$-component of the eigenvariety into the finite slope space. 
\end{theo}
\begin{proof}
By \cite[Proposition 7.5.13]{BellaicheChenevier} the restriction of $T_Y$ to $\mathcal{G}_p$ gives rise to a refined family of $\mathcal{G}_p$-representations. By Corollary $\ref{universalrefined}$ this yields a map $Y(\bar\tau)\rightarrow X(\bar\tau_p)$ and hence a map
\[f:Y(\bar\tau)\rightarrow X(\bar\tau).\]
We claim that this map is a closed immersion. Let $\nu_1:X(\bar\tau)\rightarrow \Wcal^d\times \Gbb_m$ be the morphism that is given by the projection to $\Wcal^d$ and by the evaluation of $\psi_p^{\rm gal}(u_0)^{-1}$. For an affinoid open $V\subset \Wcal^d\times \Gbb_m$ we obtain a commutative diagram
\[
\begin{xy}
\xymatrix@!R{
**[r]\Hcal\otimes_\Z\Gamma(V,\Ocal_{\Wcal^d\times\Gbb_m}) \ar[r]^{\psi\otimes\nu^\sharp}\ar[rd]_{\psi^{\rm gal}\otimes\nu_1^\sharp} & \Gamma(\nu^{-1}(V),\Ocal_{Y(\bar\tau)})\\
& \Gamma(\nu_1^{-1}(V),\Ocal_{X(\bar\tau)}).\ar[u]_{f^\sharp} 
}\end{xy}\]
In this diagram the upper arrow is surjective and hence so is $f^\sharp$. However, $\nu_1^{-1}(V)$ is a Stein space and $\nu^{-1}(V)$ is affinoid. As morphisms between quasi-Stein spaces are given by the induced maps on the global sections, this yields the claim. \end{proof}
\begin{cor}
\noindent {\rm (i)} Assume that $\bar\tau={\rm tr}\,\bar\rho$ for an absolutely irreducible representation $\bar\rho$. Then there is a family $\Vcal_{\bar\rho}$ of $\mathcal{G}_{E,S}$-representations with trace $T_Y$ on $Y(\bar\rho)=Y(\bar\tau)$. \\
\noindent {\rm (ii)} If $Z\subset Y(\bar\rho)$ is an irreducible component containing at least one point $y$ such that $\Vcal_{\bar\rho}\otimes k(y)|_{\mathcal{G}_p}$ is absolutely irreducible, then the representations on the fibers of $\Vcal$ are trianguline when restricted to $\mathcal{G}_p$.\\ 
\noindent {\rm (iii)}The family $\Vcal|_{\mathcal{G}_p}$ is trianguline over a dense Zariski-open subset of $Z$.
\end{cor}
\begin{proof}
(i) This is obvious as $\bar\rho$ is absolutely irreducible.\\
(ii), (iii) This follows from Corollary $\ref{familyofGalrep}$. We only need to see that there are points of $Z$ that map to the regular part $X(\bar\tau_p)^{\rm reg}\subset X(\bar\tau_p)$. However this follows from the density of $\Zcal$ and the fact that the points in $\Zcal$ satisfy the regularity condition. 
\end{proof}
\subsection{The $2$-dimensional case}
Before we make a conjecture concerning the cases (b) and (c) we first study case (a).  This result is already contained in Kisin's papers \cite{KisinsXfs} and \cite{KisinFM}. 
We assume that \[\bar\rho:\mathcal{G}_{\Q,S}\longrightarrow \GL_2(\Fbb)\]
is odd and absolutely irreducible when restricted to ${\rm Gal}(\Q_S,\Q(\zeta_p))$. As above we write $\bar\tau$ for the trace of $\bar\rho$. Further we assume that $\bar\rho|_{\mathcal{G}_p}$ is not an extension of a character $\psi$ by $\psi(1)$. In this case Kisin \cite{KisinFM} has proved (special cases of) the Fontaine-Mazur conjecture: A lift $\rho:\mathcal{G}_{\Q,S}\rightarrow \GL_2(\Ocal_F)$ with values in the ring of integers of a finite extension $F $ of $\Q_p$ is (up to twist) modular if it is de Rham with distinct Hodge-Tate weights at $\mathcal{G}_p$ and becomes semi-stable after restriction to the Galois group of an abelian extension of $\Q_p$. Especially the representations that are crystalline at $p$ with Hodge-Tate weights $0$ and $k>0$ are modular. 
\begin{lem}
Let $H$ be a group acting on the trivial vector bundle $\Vcal$ of rank $d$ over a noetherian scheme $X$. Then the map
\[X\ni x\longmapsto \dim_{k(x)}\big(\Vcal\otimes k(x)\big)^H\]
is upper semi-continuous. 
\end{lem}
\begin{proof}
Let us denote this map by $r$. If $x\rightsquigarrow y$ is a specialization, then it is obvious that $r(x)\leq r(y)$. We need to show that for each $m\in\Z$ the set
\[A=\{x\in X\mid r(x)\geq m\}\]
is constructible (and hence closed, as it is closed under specialization). Let us write $V=\underline\Spec_X({\rm Sym}^\bullet\Vcal^\vee)\cong \Abb^d_X$ for the geometric vector bundle associated with $\Vcal$ and let us write $B\subset V$ for the Zariski-closed subset of $H$-invariants in $V$. We consider the morphism
\[f=\bigwedge\nolimits^m:V\times_X\dots\times_XV\longrightarrow \bigwedge\nolimits^mV\cong \Abb_X^{d'}.\]
Then $C=f(B\times_X\dots\times_XB)\subset \bigwedge\nolimits^mV$ is constructible and hence so is $C\backslash 0_X$, where $0_X\subset \bigwedge^mV$ denotes the zero-section. As $A$ is the image of $C\backslash 0_X$ under the projection to $X$, it follows that $A$ is constructible which yields the claim.
\end{proof}
\begin{cor}
Let $N\geq 1$ be an integer. There is a Zariski-closed subspace $X_N(\bar\tau)$ of $X(\bar\tau)$ such that $x\in X_N(\bar\tau)$ if and only if the representation $\rho_x=\Vcal_{\bar\tau}\otimes k(x)$ has Artin conductor $N(\rho_x)$ dividing $N$. 
\end{cor}
\begin{proof}
This follows by applying the above lemma to the universal deformation on $R_{\bar\rho}$ and by the formula
\[{\rm val}_\ell(N(\rho))=\sum_{i\geq 0} \frac{\dim V-\dim V^{I_{\ell,i}}}{[I_{\ell, i}:I_{\ell}]} \]
for the $\ell$-adic valuation of the Artin conductor of a representation $\rho:\mathcal{G}_{\Q,S}\rightarrow \GL(V)$. Here $I_\ell\subset {\rm Gal}(\bar\Q_\ell/\Q_\ell)$ is the inertia group and $I_{\ell,\bullet}$ is the filtration by higher ramification groups. 
\end{proof}
Let $D$ be a definite quaternion algebra which splits at $p$ and let $Y_N$ be the eigencurve of tame level $N$ associated to $D$. Further we write $X_N^0(\bar\rho)$ for the subspace of the finite slope space $X(\bar\tau)$ where at least one Hodge-Tate weight is $0$ and where the Artin conductor is bounded by $N$. Compare \cite[Conjecture 11.8]{KisinsXfs} and \cite[Corollary]{KisinFM} for the following result. 
\begin{prop}
The inclusion $Y_N(\bar\rho)\hookrightarrow X^0_N(\bar\rho)$ is an inclusion of irreducible components. 
\end{prop}
\begin{proof}
It is enough to check that the irreducible components of $X^0_N(\bar\rho)$ where the crystalline points are dense have dimension $1$. This follows from the modularity of those representations and the fact that the eigencurve is $1$-dimensional. 
\end{proof}

\subsection{A general conjecture}
We now assume that we are in the situation (b) or (c). Especially $E$ is an imaginary quadratic extension of $\Q$. We choose a complex conjugation $c\in {\rm Gal}(\bar\Q,\Q)$ and also write $c$ for its image in ${\rm Gal}(E,\Q)$. Given a representation $\rho:\mathcal{G}_{E,S}\rightarrow \GL(V)$, we write $V^\perp$ for the dual of the representation  $g\mapsto c\rho(g)c$. The aim of this section is to give a global version of an infinitesimal conjecture of Bellaiche and Chenevier.
For a point $x=(\rho,\underline{\delta})$ we define the \emph{minimal finite slope space $X(\bar\rho,x)$ containing $x$} similar to the deformation functor of \cite[Definition 7.6.2]{BellaicheChenevier}: 
\begin{defn} Let  $x=(\rho,\underline{\delta})$. We write $X^\perp(\bar\rho)\subset X(\bar\rho)$ for the Zariski-closed subspace of $X$ where $\rho^\perp\cong \rho(d-1)$.
The minimal finite slope space $X(\bar\rho,x)$ containing $x$ is the subspace of $X^\perp(\bar\rho)$, where the action of the inertia groups $I_w$ at places $w\in S$ not dividing $p$ factors over the $I_w$-action given by $x$.
\end{defn}
This space is a closed subspace of $X(\bar\rho)$ and
for any $x\in X(\bar\rho)$ the complete local ring $\hat\Ocal_{X(\bar\rho,x),x}$ pro-represents the deformation functor defined in \cite[Definition 7.6.2]{BellaicheChenevier}.
\begin{lem}
If $x\in X(\bar\rho)(L)$ for a finite extension $L$ of $\Q_p$, then $X(\bar\rho,x)$ is a union of connected components of $X^\perp(\bar\rho)\otimes_{\Q_p}L$.
\end{lem}
\begin{proof}
It follows from \cite[Lemma 7.8.17]{BellaicheChenevier} that the restriction of the pseudo-character $T$ to the inertia groups $I_w$ is locally constant on $X$, if $w$ is a place not dividing $p$. This yields the claim. 
\end{proof}
We continue with the set up of \cite[7.6.2]{BellaicheChenevier}: Fix a prime $q\neq p$ that splits in $E$. If $d\equiv 0\mod 4$, then we fix another prime $q'\in \{p,q\}$, if not we set $q=q'$. Let $G$ be a unitary group in $d$ variables attached to $E/\Q$ such that 
\begin{enumerate}
\item[(i)] $G(\R)$ is compact.
\item[(ii)] If $\ell\notin\{q,q'\}$, then $G(\Q_\ell)$ is quasi-split.
\item[(iii)] If $\ell=q$ or $\ell=q'$, then $G(\Q_\ell)$ is the group of invertible elements of a central division algebra over $\Q_\ell$.
\end{enumerate}
Such a group exists by \cite[Lemma 7.6.8]{BellaicheChenevier}. 
Let $\pi$ be an automorphic representations of $G$ such that $\pi$ is only ramified at primes that split in $E$ and such that $\pi_q$ is supercuspidal. Further we assume that $\pi_p$ is unramified and its Satake-parameter has $d$ distinct eigenvalues. We denote by $Y_{\pi}$ the minimal eigenvariety containing $\pi$ and write $z\in Y_{\pi}$ for the classical points defined by $\pi$, compare \cite[7.6.2]{BellaicheChenevier}. 
\begin{lem}
The inclusion $Y_\pi(\bar\rho)\rightarrow X(\bar\rho)$ factors over $X(\bar\rho,z)$. \end{lem}
\begin{proof}
The argument is similar to the one in \cite[7.6.2]{BellaicheChenevier}.
It follows from the theorem of Harris Taylor \cite{HarrisTaylor} that there are Galois representations $\rho_\pi$ attached to most classical points $\pi$  and that they satisfy the condition $\rho_{\pi}^\perp\cong \rho_\pi(d-1)$. \end{proof}
Let $\rho_\pi$ denote the representation of $\mathcal{G}_{E,S}$ associated with $\pi$. As we assume that $\bar\rho$ is absolutely irreducible and that the Hodge-Tate weights $k_1<k_2<\dots<k_d$ of $\rho_\pi|_{\mathcal{G}_p}$ are pairwise distinct, it follows from Lemma $\ref{triangulineGpfamily}$ that there is a canonical triangulation\footnote{This allows us to drop the regularity condition in $(\rho_4)$ of \cite[7.6.2]{BellaicheChenevier}} $\Fil_i(D)$ of the $(\phi,\Gamma)$-module $D={\bf D}_{\rig}^\dagger(\rho_\pi)$.  We assume that $\rho_\pi|_{\mathcal{G}_p}$ is \emph{non critically refined}, i.e. that 
\[D_{\rm cris}(\rho)=({\rm Fil}_i(D)[1/t])^\Gamma\oplus \Fil^{k_i+1}(D_{\rm cris}(\rho))\]
for all $1\leq i\leq d$. This means that the trianguline filtration is in general position with respect to the Hodge-filtration. 
\begin{conj}
Let $\pi$ be as above. The map $Y_\pi(\bar\rho)\rightarrow X(\bar\rho,z)$ is an inclusion of irreducible components of $X(\bar\rho,z)$. 
\end{conj} 
Bellaiche and Chenevier conjecture the infinitesimal analogue of our conjecture. Namely they conjecture that $\hat \Ocal_{Y_\pi,z}\cong \hat \Ocal_{X(\bar\rho,z),z}$, see \cite[Conjecture 7.6.12]{BellaicheChenevier}.
As our conjecture comes down to a computation of the dimensions of the irreducible components of $X(\bar\rho,z)$ we see that it is equivalent to conjecture that 
\[\hat \Ocal_{Y_\pi,z'}\cong \Ocal_{X(\bar\rho,z),z'}\]
holds for at least one point $z'$ in each irreducible component of $Y_{\pi}$.
\begin{rem}
 Bellaiche and Chenevier also show that this conjecture is implied by the Bloch-Kato conjecture. Hence it is a very safe conjecture. 
 \end{rem}

\bigskip

\end{document}